\documentclass[11pt]{amsart}
\usepackage{amssymb,latexsym}
\usepackage[utf8]{inputenc}
\usepackage{amsmath,amsthm}
\usepackage{amsfonts,mathrsfs}
\usepackage{tikz}
\usetikzlibrary{arrows.meta}
\usepackage[margin=1in]{geometry}
\usepackage{float}
\usetikzlibrary{calc}

\usepackage{enumerate,units}
\usepackage[all]{xy}
\usepackage{graphicx}

\usepackage{hyperref, url}
   
\theoremstyle{plain}
\newtheorem{theorem}{Theorem}[section]
\newtheorem{corollary}[theorem]{Corollary}
\newtheorem{lemma}[theorem]{Lemma}
\newtheorem{proposition}[theorem]{Proposition}

\newtheorem{fact}[theorem]{Fact}
\newtheorem*{claim}{Claim}
\newtheorem*{theorem*}{Theorem}
\newtheorem*{proposition*}{Proposition}

\theoremstyle{definition}
\newtheorem{definition}[theorem]{Definition}
\newtheorem{example}[theorem]{Example}

\theoremstyle{remark}
\newtheorem{remark}[theorem]{Remark}

\newtheorem{question}[theorem]{\textbf{Question}}

\numberwithin{equation}{section}
\newcommand{\forkindep}[1][]{%
  \mathrel{
    \mathop{
      \vcenter{
        \hbox{\oalign{\noalign{\kern-.3ex}\hfil$\vert$\hfil\cr
              \noalign{\kern-.7ex}
              $\smile$\cr\noalign{\kern-.3ex}}}
      }
    }\displaylimits_{#1}
  }
}

\newenvironment{claimproof}[1][\proofname]
  {%
    \proof[#1]%
  }
  {%
    \endproof%
  }

\newcounter{step}                   
    {\hfill $\clubsuit$             
     \vspace{7pt}\par}

\DeclareMathOperator{\acl}{acl}

\DeclareMathOperator{\tp}{tp}

\begin{document}
\title{Model-theoretic Tameness in finite extensions of groups}

\author{Yatir Halevi}
\address{Faculty of Mathematics\\Technion - Israel Institute of Technology\\Haifa\\ Israel.}
 \email{yatirh@math.technion.ac.il}

\author{Saharon Shelah}
\address{Einstein Institute of Mathematics\\Hebrew University of Jerusalem\\91904\\Jerusalem\\
Israel.}
\email{shelah@math.huji.ac.il}

\thanks{The second author thanks ISF 2320/23: The Israel Science Foundation (ISF) (2023-2027) for partially supporting this research. Paper 1274 in the publication list of Shelah.}

\begin{abstract}
It is shown that finite-index extensions and finite-index subgroups of $\omega$-stable groups can be model-theoretically wild. More precisely, there exists an $\omega$-stable group $G$ such that any given countable first-order structure in a finite language is interpretable both in some finite-index extension of $G$ and in some finite-index subgroup of $G$.
\end{abstract}

\maketitle
\section{Introduction}
A basic and recurring structural question in group theory is: Which properties are inherited or transferred in finite group extensions?

In the case of model-theoretic stability, the question is:
\begin{question}
    Let $G$ be an infinite group, and let $H\leq G$ be a subgroup of finite index. Is stability preserved when passing between $G$ and $H$? More precisely, if $G$ is stable, must $H$ be stable? Conversely, if $H$ is stable, must $G$ be stable?
\end{question}
One can, of course, ask this for any model-theoretic tameness notion.

The direction of transferability of stability up in finite extensions is known for abelian groups: abelian-by-finite groups are stable (folklore, essentially due to \cite{BaChMa}). The general question in this direction goes back at least to \cite{Meirembekov}. A counterexample was given by Simon Thomas and was communicated by Baldwin in \cite[Example 3.2]{baldwin}: The semidirect product of $\mathrm{SL}_2(\mathbb{C})$ and $\mathbb{Z}/2\mathbb{Z}$ under the action of complex conjugation.\footnote{We thank Udi Hrushovski for independently pointing this example out to us.}


We answer the general question negatively. The situation can be made as bad as possible with respect to model-theoretic tameness.

\begin{theorem}[Proposition \ref{P: down} and \ref{P: up}]\label{T:main}
    Let $p$ be an odd prime. There exists an $\omega$-stable group $H$ of nilpotency class $2$ and exponent $p$ such that for any first-order countable structure $M$ in a finite  language:
    \begin{enumerate}
        \item   There is a supergroup $G\supseteq H$ satisfying that $[G:H]=2$ and that $(G,\cdot)$ interprets $M$.

        \item There is a normal subgroup $K\trianglelefteq H$ of index $p$ such that $(K,\cdot)$ interprets $M$.
    \end{enumerate}
\end{theorem}
\begin{remark}
\begin{enumerate}
    \item This implies that the folklore result on the stability of abelian-by-finite groups is strict in the following sense. Replacing ``abelian'' by ``solvable'' does not guarantee stability (this is already implied by a variant of Simon Thomas's counterexample). As for the other direction, if $H$ is a perfect group with bounded commutator length and $G/H$ is finite abelian (or even solvable) then it is easy to see that $H$ is definable in $G$. 
    
    \item The countability assumption is not essential here. One can find appropriate groups to interpret structures of large cardinality as well.
\end{enumerate}
\end{remark}

The group is produced by applying Mekler's construction \cite{Mek:stability}.  The desired supergroup is obtained by producing an automorphism, and the subgroup is obtained with the aid of an appropriate linear functional on the $\mathbb{F}_p$-vector space $H/Z(H)$.

In the appendix we prove that certain classical groups are not superstable. Although the results there follow from \cite[Corollary 1.3]{superstablereisdual} where it is shown that a finitely generated residually finite group is superstable if and only if it is abelian-by-finite, we think that the proof presented here might be of independent interest. The proof is an abstraction of Poizat's proof that the free group is not superstable \cite{Poi:FreeNSuperstable}.

All our model-theoretic notation and notions are standard; the reader may consult \cite{Hodges} when needed. Unless specified otherwise, by ``definable'' and ``interpretable'' we mean with parameters.

\vspace{0.2cm}
\noindent\emph{\textbf{Acknowledgement.}} We would like to thank Chen Meiri and Gregory Cherlin for useful conversations and suggestions and Immi Halupczok
for suggesting a simplification of the graph.

\section{Mekler's Construction}
We review the basic properties of Mekler groups we will need.

\begin{definition}\cite[Definition 1.1]{Mek:stability}
    A \emph{nice graph} (in the sense of Mekler) is a graph $\Gamma=(V,E)$ satisfying:
    \begin{enumerate}
        \item The graph has at least two vertices and for every two distinct vertices $v,u\in V$ there is some vertex, not equal to $v$ or $u$, that is connected to $v$ but not to $u$.
        \item There are no triangles or squares in the graph.
    \end{enumerate}
\end{definition}

By \cite[Definition 2.1]{Mek:stability} (see also \cite[Appendix A.3]{Hodges}) to each nice graph $\Gamma$  and odd prime $p$, one associates a nilpotent group of class $2$ and  exponent $p$,  $H=H(\Gamma)$. It is the free group in the variety of nilpotent groups of class $2$ and  exponent $p$ generated by $\{x_v: v\in V\}$ whose only relations impose that $x_v$ and $x_{v'}$ commute if and only if $(v,v')\in E$.

Note that the definition readily implies that every automorphism of the graph $\Gamma$ induces an automorphism of the group $H$.

We review some important properties of the group; we will use them without further reference.

\begin{fact}\label{F:properties of Mekler group}\cite[Appendix A.3]{Hodges}
    \begin{enumerate}
        \item $Z(H)=[H,H]$.
        \item $Z(H)$ is an elementary abelian $p$-group (so a $\mathbb{F}_p$-vector space) with basis $\{[x_v,x_u]: v<u\in V,\, (v,u)\notin E\}$, for some fixed linear order $<$ on $V$
        \item $H/Z(H)$ is an elementary abelian $p$-group  with basis $\{x_vZ(H):v\in V\}$.
    \end{enumerate}
\end{fact}

\begin{fact}\label{F: stability}\cite[Theorem 2.12]{Mek:stability}
    The stability spectrum of $\Gamma$ (as a graph) is equal to that of $H=H(\Gamma)$ (as a group).
\end{fact}

Every element in $H\setminus Z(H)$ is of the form $x_{v_1}^{\alpha_1}\dots x_{v_k}^{\alpha_k}z$, for $v_1,\dots,v_k\in V$, $0<\alpha_i<p$ and $z\in Z(H)$. The set $\{v_1,\dots,v_k\}$ is called the support of that element, and $k$ is called its length.

We borrow terminology from the graph $\Gamma$: an element of the form $x_v$ is called a vertex, and we say that $x_v$ and $x_u$ are neighbors if $(v,u)\in E$.

\begin{lemma}\label{L: dim computations}
    Let $a\in H\setminus Z(H)$, and let $S$ be the support of $a$.
    \begin{enumerate}
        \item If $|S|=1$ and the only vertex in the support of $a$ has infinitely many neighbors then $\dim_{\mathbb{F}_p}C_H(a)/Z(H)$ is infinite.
        \item Assume there is an integer $n_0$ such that if $v\in V$ has finitely many neighbors then it has at most $n_0$ neighbors. Further, assume that if $|S|=1$ then the unique vertex in the support of $a$ has finitely many neighbors. Then $\dim_{\mathbb{F}_p}C_H(a)/Z(H)\leq \max\{3,n_0+1\}$.
    \end{enumerate}
\end{lemma}
\begin{proof}
    (1) Assume that $x_v$ is the unique element in the support of $a$. Then $x_u\in C_H(a)$ for any $u\in V$ with $(u,v)\in E$.

    (2) Since $H/Z(H)$ is a vector space over $\mathbb{F}_p$ it will be convenient to write everything additively. The commutator map $[\cdot,\cdot]$ induces a bilinear map $H/Z(H)\times H/Z(H)\to Z(H)$ (recall that the latter is also an $\mathbb{F}_p$-vector space). Let $a=\alpha_1 x_{v_1}+\dots + \alpha_k x_{v_k}\in H/Z(H)$ and set $S=\{v_1,\dots, v_k\}$ to be its support. Let $b\in C_H(a)/Z(H)$ be a non-central element. Write $b=b_0+b_1$ where the vertices in $b_0$ are from the support of $S$, and those in $b_1$ are disjoint from $S$. 

    \begin{claim}
        For $T=\{x_v: \text{$(v,u)\in E$ for all $u\in S$}\}$,
        $b_1\in \mathrm{Sp}_{\mathbb{F}_p}(T)$. If $|S|>1$ then $|T|\leq 1$. 
    \end{claim}
    \begin{claimproof}
        Let $x_v$ be a vertex in the support of $b_1$ and let $\beta$ be its coefficient. As $x_v$ is not in $S$, for any $v_i\in S$ the coefficient of $[x_v,x_{v_i}]$ in $[b,a]=0$ (as an element of $Z(H)$) is $\beta\alpha_i$. If $(v,v_i)\notin E$ then $p|\beta\alpha_i$; but $0,\beta,\alpha_i<p$, contradiction.

        Assume that $|S|>1$, and let $s,t\in S$ be distinct. Towards a contradiction let $v,u$ be distinct vertices connected to both $s$ and $t$. If $t$ and $s$ are connected we get a triangle. If $t$ and $s$ are not connected we arrive to a square. Either way, this contradicts the niceness of $\Gamma$.
    \end{claimproof}

    Assume that $|S|=1$ and that the vertex in the support of $a$ has only finitely many neighbors, then the length of $b_0$ is at most $1$ and $|T|$ is bounded above by $n_0$. Thus every such $b$ lies inside a $(n_0+1)$-dimensional subspace, so $\dim_{\mathbb{F}_p} C_H(a)/Z(H)\leq n_0+1$.
    
    We now assume that $|S|>1$, so $|T|\leq 1$. By the claim, $b_1$ is in the span of $T$, so lies in a subspace of dimension at most $1$. It remains to consider $b_0$. We will show that $b_0$ lies in a subspace of dimension $2$ and so in this case $\dim_{\mathbb{F}_p}C_H(a)/Z(H)\leq 3$ independently of the choice of $a$.

     Let $\Delta$ be the complement graph induced on the finite set of vertices $S$. I.e., two elements in $S$ are connected by an edge if they are not connected by an edge in $\Gamma$. The point is that $\Delta$ has at most $2$ connected components. Indeed, otherwise $\Gamma$ would contain a triangle. Let $C_1$ and $C_2$ be two connected components; if $\Delta$ has only one connected component, set $C_2=\emptyset$.

        For the following, recall that $a=\alpha_1x_{v_1}+\dots+ \alpha_k x_{v_k}$. Let $a_1=\sum_{v_i\in C_1}\alpha_iv_i$ and $a_2=\sum_{v_i\in C_2}\alpha_iv_i$.

        We claim that $b_0\in \mathrm{Sp}_{\mathbb{F}_p}\{a_1,a_2\}$. Assume that $b_0=\beta_1x_{v_1}+\dots+\beta_kx_{v_k}$, where some of the $\beta_i$ may be equal to $0$. As $[a,b]=0$, for every $1\leq i<j\leq k$ if $[x_{v_i},x_{v_j}]\neq e$ then $\frac{\beta_i}{\alpha_i}=\frac{\beta_j}{\alpha_j}$ modulo $p$ (see \cite[Lemma A.3.4]{Hodges}). This means that there are $c_1,c_2\in \mathbb{F}_p$ such that for every $v_i\in C_1$, $\beta_i=c_1\alpha_i$ and for every $v_i\in C_2$, $\beta_i=c_2\alpha_i$. The result follows easily.

        We conclude that $\dim_{\mathbb{F}_p}C_H(a)/Z(H)\leq \max\{3,n_0+1\}$.
\end{proof}

\begin{definition}
    An element of $H$ is \emph{vertex-like} if it is of the form $x_v^\alpha z$, for $v\in V$, $0<\alpha<p$ and $z\in Z(H)$.
\end{definition}

    By \cite[Theorem A.3.10(b)]{Hodges}, the set of vertex-like elements is $\emptyset$-definable. Thus, to prove the following result, we do not really require the previous proposition. The previous proposition will be useful in the sequel.
\begin{corollary}
    Assume there is an integer $n_0$ such that if $v\in V$ has finitely many neighbors then it has at most $n_0$ neighbors. Then the set of vertex-like elements with infinitely many neighbors, that is, \[\{x_v^{\alpha}z: 0<\alpha<p,\, z\in Z(H),\, \text{$v$ has infinitely many neighbors}\},\] is $\emptyset$-definable.
\end{corollary}

We now define the main graph that will be of interest to us. 
Let $A=(V,E)$ be the graph with vertex set $V=\mathbb{N}\cup\{(\{a,b\},i): a\neq b\in \mathbb{N},\, i=0,1,1.25,1.5,1.75\}$ and edge relation $E$ which is the symmetric closure of \[\{(n,(\{a,b\},0)): a,b,n\in \mathbb{N},\,a\neq b,\, a=n\vee b=n\}\cup \]\[\{((\{a,b\},0),(\{a,b\},1)): a\neq b\in \mathbb{N}\}\cup\]
\[\{((\{a,b\},1),(\{a,b\},1.25)): a\neq b\in \mathbb{N}\}\cup\]
\[\{((\{a,b\},1.25),(\{a,b\},1.5)): a\neq b\in \mathbb{N}\}\cup\]
\[\{((\{a,b\},1.5),(\{a,b\},1.75)): a\neq b\in \mathbb{N}\}\cup\]
\[\{((\{a,b\},1.75),(\{a,b\},0)): a\neq b\in \mathbb{N}\}.\]

The following is a rough illustration.

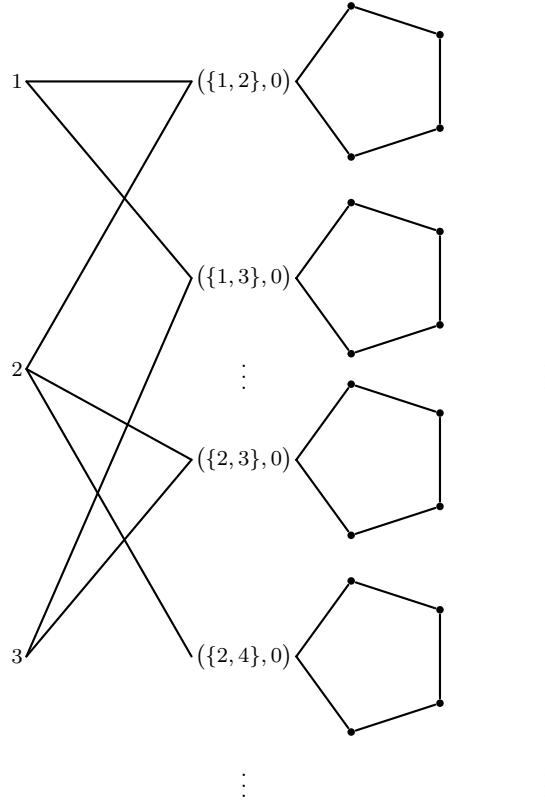
\begin{figure}[H]
\centering
\begin{tikzpicture}[
    thick,
    line cap=round,
    dot/.style={circle, fill, inner sep=1pt},
    every node/.style={font=\scriptsize, inner sep=1pt}
]

\def\PentR{1.05}

\newcommand{\pentagoncycle}[2]{%
    \begin{scope}[shift={(#1)}]
        \coordinate (#2o) at (0,0);

        \node[dot] (#2a) at ({\PentR*(1+cos(108))},{\PentR*sin(108)}) {};
        \node[dot] (#2b) at ({\PentR*(1+cos(36))}, {\PentR*sin(36)}) {};
        \node[dot] (#2c) at ({\PentR*(1+cos(-36))},{\PentR*sin(-36)}) {};
        \node[dot] (#2d) at ({\PentR*(1+cos(-108))},{\PentR*sin(-108)}) {};

        \draw
            (#2o) --
            (#2a) --
            (#2b) --
            (#2c) --
            (#2d) --
            (#2o);
    \end{scope}
}

\node (n1) at (0,  3.8) {$1$};
\node (n2) at (0,  0.0) {$2$};
\node (n3) at (0, -3.8) {$3$};

\node (ab12) at (3.0,  3.8) {$\bigl(\{1,2\},0\bigr)$};
\node (ab13) at (3.0,  1.2) {$\bigl(\{1,3\},0\bigr)$};
\node (ab23) at (3.0, -1.2) {$\bigl(\{2,3\},0\bigr)$};
\node (ab24) at (3.0, -3.8) {$\bigl(\{2,4\},0\bigr)$};

\draw (n1.east) -- (ab12.west);
\draw (n1.east) -- (ab13.west);

\draw (n2.east) -- (ab12.west);
\draw (n2.east) -- (ab23.west);
\draw (n2.east) -- (ab24.west);

\draw (n3.east) -- (ab13.west);
\draw (n3.east) -- (ab23.west);

\node at (3.0, 0.0) {$\vdots$};
\node at (3.0,-5.4) {$\vdots$};


\pentagoncycle{ab12.east}{i12}
\pentagoncycle{ab13.east}{i13}
\pentagoncycle{ab23.east}{i23}
\pentagoncycle{ab24.east}{i24}

\node at (7.0, 0.0) {$\vdots$};
\node at (7.0,-5.4) {$\vdots$};

\end{tikzpicture}
\caption{For each pair \(\{a,b\}\), the pentagon represents the 5-cycle attached to \((\{a,b\},0)\).}
\end{figure}

As $A$ is interpretable in the theory of infinite sets, $A$ is $\omega$-stable. It is straightforward to verify that $A$ is a nice graph and that $n_0=4$ is the finite-degree bound. Let $H=H(A)$ be the corresponding Mekler group. By Fact \ref{F: stability}, $H$ is $\omega$-stable as well.

\section{Going Up}

The goal of this section is to show  that every structure in a finite relational language can be interpreted in some finite extension of $H$ (even of degree $2$).

Let $\Gamma=(\mathbb{N},R)$ be a graph. We will also regard $R$ as a subset of $\{\{a,b\}: a\neq b\in \mathbb{N}\}$, let $\widetilde \pi_R$ be the graph automorphism swapping, for any $\{a,b\}\in R$,
\begin{list}{$\bullet$}{}
    \item $(\{a,b\},1)\leftrightarrow(\{a,b\},1.75)$ and
    \item $(\{a,b\},1.25)\leftrightarrow(\{a,b\},1.5)$,
\end{list}
and fixing all other vertices. It induces an automorphism of $H$ which we denote by $\pi_R$.

We now work in the structure $(H,\cdot,\pi_R)$. 
Let $\psi(x,y)$ be the formula stating that there does not exist an element $z\in Z(H)$ and integer $0<\alpha<p$ for which $x^\alpha=yz$.
Recall that the set of vertex-like elements is $\emptyset$-definable. Let $\varphi_R(x,y)$ be the formula \[\psi(x,y)\wedge  \exists u\exists v(\text{`$u$ and $v$ are vertex-like' }\wedge\pi_R(v)Z(H)\neq vZ(H)\]\[\wedge [u,x]=[u,y]=[u,v]=e).\]
The following is the crucial lemma.

\begin{lemma}
    For any $n, m\in \mathbb{N}$, $0<\gamma,\delta<p$ and $c,d\in Z(H)$, \[(H,\cdot ,\pi_R)\models \varphi_R(x_n^\gamma c,x_m^\delta d)\iff \{n,m\}\in R.\]
\end{lemma}
\begin{proof}
    The first direction is straightforward: For $\{n,m\}\in R$ necessarily $n\neq m$ and we can choose $u=x_{(\{n,m\},0)}$ and $v=x_{(\{n,m\},1)}$.

    Assume that $\varphi_R(x_n^\gamma c,x_m^\delta d)$ holds in $(H,\cdot ,\pi_R)$. As $\psi(x_n^\gamma c,x_m^\delta d)$ necessarily $n\neq m$. Let $u,v$ be elements witnessing that the formula holds: so $u=x_t^{\alpha}z$ and $v=x_s^{\beta}z'$ for $t,s\in V$, $0< \alpha,\beta< p$ and $z,z'\in Z(H)$. 

    Since $[x_t^{\alpha}z,x_n^\gamma c]=[x_t^{\alpha}z,x_m^\delta d]=e$, we also get $[x_t,x_n]^{\alpha \gamma}=[x_t,x_m]^{\alpha \delta}=e$, but since $p$ does not divide $\alpha\gamma$ and $\alpha\delta$ necessarily $(t,m),(t,n)\in E$. This means that $t=(\{n,m\},0)$. As $\pi_R(v)Z(H)\neq vZ(H)$, we get that $s=(\{a,b\},i)$ for $\{a,b\}\in R$ and $i=1,1.25,1.5$ or $1.75$. 
    
    We also have $[x_t^{\alpha}z,x_s^{\beta}z']=e$, so $[x_t,x_s]^{\alpha\beta}=e$. If $x_t$ and $x_s$ do not commute, then  $p|\alpha\beta$, contradiction. We deduce that $\{n,m\}\in R$.
\end{proof}

\begin{corollary}\label{C: up interpret any countable graph}
  The countable graph $\Gamma=(\mathbb{N},R)$ is interpretable in  $(H,\cdot,\pi_R)$.
\end{corollary}
\begin{proof}
 Let $Y$ be the set of vertex-like elements with infinitely many neighbors, and let $\approx$ be the equivalence relation defined by: $g\approx h \iff \text{modulo $Z(H)$,  $g$ and $h$ are powers of the other}$. Then $Y/\approx$ can be identified with the vertices of the form $x_n$, $n\in \mathbb{N}$. The formula $\varphi_R$ induces the edge relation $R$ on $Y/\approx$.    
\end{proof}


\begin{proposition}\label{P: up}
For any first-order countable structure $M$ in a finite  language there is  a supergroup $G\supseteq H$ such that $[G:H]=2$ and that $(G,\cdot)$ interprets $M$.
\end{proposition}
\begin{proof}
    By, for example, \cite[Theorem 5.5.1]{Hodges}, we may interpret the structure $M$ in a countable graph $\Gamma=(\mathbb{N},R)$. Thus, $\Gamma$ is interpretable in $(H,\cdot,\pi_R)$.

    Let $G=H\rtimes_{\pi_R} C_2$ with $C_2=\{0,1\}$, i.e. consider the set $G=H\times C_2$ with group structure $(h,\varepsilon)\cdot (k,\delta)=(h\pi_R^\varepsilon(k),\varepsilon+\delta)$, where $\pi_R^0=\mathrm{id}$ and $\pi_R^1=\pi_R$.

    Inside $G$, $H$ is definable by the formula $x^p=e$. For $t=(e,1)$, the automorphism $\pi_R$ is given on $H$ by $h\mapsto tht^{-1}$. Hence, $(H,\cdot,\pi_R)$ is definable in $G$, and $M$ is interpretable in $G$.
\end{proof}

\section{Going Down}

The goal of this section is to show that every structure in a finite relational language can be interpreted in some finite-index subgroup of $H$.

As $H/Z(H)$ is an $\mathbb{F}_p$-vector space with basis $\{x_vZ(H):v\in V\}$,  for any countable graph $\Gamma=(\mathbb{N},R)$ we define the following linear functional $\ell_R:H/Z(H)\to \mathbb{F}_p$:
\[
 \ell_R(x_vZ(H))=
 \begin{cases}
 0,&\text{if $v\in \mathbb{N}$},\\
 0,&\text{if $v=(\{n,m\},0)$ and $\{n,m\}\in R$},\\
 1,&\text{otherwise.}
 \end{cases}
\]

The kernel $\ker(\ell_R)$ is an abelian subgroup of $H/Z(H)$ so $H_\Gamma:=\pi^{-1}(\ker(\ell_R))$, where $\pi:H\to H/Z(H)$, is a normal subgroup of $H$ of index $p$.

Note that $x_n\in H_\Gamma$ for all $n\in \mathbb{N}$ and $x_{(\{n,m\},0)}\in H_\Gamma$ for all $\{n,m\}\in R$.

\begin{lemma}\label{L:auxiallry going down}
    \begin{enumerate}
        \item For any $x_n\in H$, $0<\alpha<p$, $c\in Z(H)$ and $a=x_{t_1}^{\alpha_1}\cdot\ldots\cdot x_{t_k}^{\alpha_k}z\in H$ with $t_i\in V$ distinct, $0<\alpha_i<p$ and $z\in Z(H)$: \[[x_n^\alpha c,a]=e \iff \text{ for every $i$, either $t_i=n$ or there is some $l\in \mathbb{N}$, $l\neq n$, with  $t_i=(\{n,l\},0)$}.\]
        \item $Z(H)=Z(H_\Gamma)$.
    \end{enumerate}
\end{lemma}
\begin{proof}
    (1) If $[a,x_n^\alpha c]=e$ then $\prod_i [x_{t_i},x_n]^{\alpha\alpha_i}=e$. Writing $I=\{i: [x_{t_i},x_n]=e\}$, we have $\prod_{i\notin I}[x_{t_i},x_n]^{\alpha \alpha_i}=e$, hence $p|\alpha\alpha_i$. But $0<\alpha,\alpha_i<p$ so $I=\{i: 1\leq i\leq k\}$. The other direction is straightforward.

    (2) Let $a\in Z(H_\Gamma)$. Since $x_n \in H_\Gamma$ for all $n\in \mathbb{N}$, Part (1) implies that $a\in Z(H)$. For the other direction, note that $Z(H)\subseteq H_\Gamma$ by definition and so $Z(H)\subseteq Z(H_\Gamma)$.
\end{proof}

From now on we identify between $Z(H)$ and $Z(H_\Gamma)$. As before, let  $\psi(x,y)$ be the formula stating that there does not exist an element $z\in Z(H)$ and integer $0<\alpha<p$ for which $x^\alpha=yz$.
Let $\varphi(x,y)$ be the formula \[\psi(x,y)\wedge \exists v(v\notin Z(H)\wedge [v,x]=[v,y]=e).\]

\begin{lemma}\label{L: down relation}
    For any $n,m\in \mathbb{N}$, $0<\gamma,\delta<p$ and $c,d\in Z(H_\Gamma)$, \[(H_\Gamma,\cdot)\models \varphi(x_n^\gamma c,x_m^\delta d)\iff \{n,m\}\in R.\]
\end{lemma}
\begin{proof}
    If $\{n,m\}\in R$ then $n\neq m$ and $x_{(\{n,m\},0)}\in H_\Gamma$. As $[x_{(\{n,m\},0)},x_n^\gamma c]=[x_{(\{n,m\},0)},x_m^\delta d]=e$ and $x_{(\{n,m\},0)}\notin Z(H)=Z(H_\Gamma)$, necessarily $H_\Gamma\models \varphi(x_n\gamma c,x_m^\delta d)$.

    For the other direction, assume that $v=x_{t_1}^{\alpha_1}\cdot\ldots\cdot x_{t_k}^{\alpha_k}z\in H_\Gamma$ witnesses $H_\Gamma\models \varphi(x_n^\gamma c,x_m^\delta d)$, with $z\in Z(H)$, $t_i\in V$ and $0<\alpha_i<p$ for all $i$. Since $v\notin Z(H_\Gamma)=Z(H)$, $k\geq 1$.

    Since $\psi(x_n^\gamma c, x_m^\delta d)$ holds, necessarily $x_n\neq x_m$. Using Lemma \ref{L:auxiallry going down}, $[v,x_n^\gamma c]=[v,x_m^\delta d]=e$ implies that $k=1$ and $t_i=(\{n,m\},0)$.

    Since $v=x_{(\{n,m\},0)}^{\alpha}z\in H_\Gamma$, $\ell_R(x_{(\{n,m\},0)}^{\alpha}Z(H))=0$. So either $\{n,m\}\in R$, as required or $\alpha=0$, contradiction.
\end{proof}

\begin{lemma}\label{L:down vertex n}
    \begin{enumerate}
        \item For any $a=x_n^\alpha c\in H_\Gamma\setminus Z(H)$, $n\in \mathbb{N}$, $0<\alpha<p$, $c\in Z(H)$, \[\dim_{\mathbb{F}_p}C_{H_\Gamma}(a)/Z(H) \geq 6.\]
        
        \item The set of elements of the form $x_n^\alpha c\in H_\Gamma$ is $\emptyset$-definable in $(H_\Gamma,\cdot)$.
    \end{enumerate}
\end{lemma}
\begin{proof}
    (1) Among the countably many values of  $\ell_R(x_{(\{n,m\},0)})\in \{0,1\}$, $n\neq m$, one value occurs infinitely often. 
    
    If $\ell_R(x_{(\{n,m\},0)})=0$ for infinitely many $m$, then we can find distinct $\{m_i\in \mathbb{N}:1\leq i\leq 6\}$ and not equal to $n$ for which $\ell_R(x_{(\{n,m_i\},0)})=0$, for $1\leq i\leq 6$. Hence, $x_{(\{n,m_i\},0)}\in C_{H_\Gamma}(a)$ for $1\leq i\leq 6$ and their images in $H_\Gamma/Z(H)$ are linearly independent.

    If $\ell_R(x_{(\{n,m\},0)})=1$ for infinitely many $m$ then we can find $\{m_i\in \mathbb{N}: 0\leq i\leq 6\}$ distinct and not equal to $n$ for which $\ell_R(x_{(\{n,m_i\},0)})=1$, for $0\leq i\leq 6$. Now  consider $x_{(\{n,m_i\},0)}(x_{(\{n,m_0\},0)})^{-1}$, $1\leq i\leq 6$, and proceed as above.

(2) By Lemma \ref{L: dim computations}(2), for every $a\in H_\Gamma\setminus Z(H)$ that is not of the form $x_n^\alpha c$, $\dim_{\mathbb{F}_p}C_{H_\Gamma}(a)/Z(H)\leq \dim_{\mathbb{F}_p}C_H(a)/Z(H)\leq 5$. Combining this with Part (1), we obtain the desired defining formula. 
\end{proof}

\begin{corollary}
    The countable graph $\Gamma=(\mathbb{N},R)$ is interpretable in $(H_\Gamma,\cdot)$.
\end{corollary}
\begin{proof}
    By Lemma \ref{L:down vertex n}(2), the set $Y=\{x_n^\alpha c:0<\alpha<p,\, n\in\mathbb{N},\, c\in Z(H)\}$ is definable in $(H_\Gamma,\cdot)$. Let $\approx$ be the equivalence relation as in the proof of Corollary \ref{C: up interpret any countable graph}. Lemma \ref{L: down relation} implies that $\Gamma$ is definable in $Y/\approx$.
\end{proof}

Using \cite[Theorem 5.5.1]{Hodges} (as in the proof of Proposition \ref{P: up}) we conclude:

\begin{proposition}\label{P: down}
   For any first-order countable structure $M$ in a finite language there is a normal subgroup $H_\Gamma\trianglelefteq H$  of index $p$ such that $(H_\Gamma,\cdot)$ interprets $M$.
\end{proposition}

\begin{question}
    Can one find an example with a subgroup of index $2$?    
\end{question}

\appendix
\section{Non-superstability}
Showing that a certain group is model-theoretically tame is usually quite difficult; showing that it is not is sometimes easier. 

There is a non-exhaustive list of results showing that certain familiar groups are not superstable; see \cite{Poi:FreeNSuperstable, MuGiSh,superstablereisdual, CaPaPa}.

Our aim here is to prove that certain groups commensurable with classical groups are not superstable by mimicking Poizat's proof for the free group. Although the proof is different, the results are not new since they follow from \cite{superstablereisdual}.

\begin{definition}\label{D: property Q}
    For a group $G$ and integers $n\geq 2$, $m\geq 1$, define:
    \begin{enumerate}
        \item $A_{n,m}(G)=\{g\in G: \text{$g^n$ has at most $m$ $n$-th roots in $G$}\}$.
        \item $G^n=\{g^n: g \in G\}$.
    \end{enumerate}
    These are definable sets in the group language. We say that a group satisfies property $Q$ if there are integers $n\geq 2$, $m\geq 1$ such that $A_{n,m}(G)$ is a generic\footnote{That is, finitely many translates of it cover $G$.} set in $G$ but $G^n$ is not.
\end{definition}

Property $Q$ is precisely what is needed for Poizat's proof of non-superstability of the free group. The following is essentially his proof and is well known.

\begin{proposition}\label{P:Prop Q is not superstable}
    Any group satisfying property $Q$ is not  superstable.  
\end{proposition}
\begin{proof}
    Assume, toward a contradiction, that $G$ is superstable and satisfies property $Q$.  As $G^n$ is not generic, $G\setminus G^n$ must be a generic definable set by stability.

    By stability theory, each generic set contains a generic type. Let $\tp(a)$, for $a\in \mathbb{G}\succ G$ in a sufficiently saturated extension, be a generic type concentrated on the definable set $A_{n,m}(G)$. Since $G$ is assumed to be superstable and $a\in \acl(a^n)$, $\tp(a^n)$ is also a generic type \cite[Lemma 5.16]{Poi:StableGroups}. Since $\tp(a^n)$ can contain only generic definable sets it must concentrate on $G\setminus G^n$, contradiction.
\end{proof}

The proposition brings the properties from Definition \ref{D: property Q} to the forefront. The following are easy observations. 

\begin{lemma}\label{L:basic property Q}
    Let $G$ be a group and $H\leq G$ a finite-index subgroup.
    \begin{enumerate}
        \item If $H^n$ is generic in $H$ then $G^n$ is generic in $G$.
        \item If $G\subseteq A_{n,m}(G)$ for $n\geq 2,\, m\geq 1$ then $H\subseteq A_{n,m}(H)$.
        \item If $H\trianglelefteq G$ and for some prime number $p>|G/H|$, $H\subseteq A_{p,m}(H)$  then $A_{p,m}(G)$ is generic in $G$.
    \end{enumerate}
\end{lemma}
\begin{proof}
    (1,2) Straightforward.    
    (3) Since $H\trianglelefteq G$ and $p>|G/H|$, for $x\in G$, if $x^p\in H$ then $x\in H$; we conclude that $A_{p,m}(H)\subseteq A_{p,m}(G)$. Hence, $G=\bigcup_i t_i H\subseteq \bigcup_i t_iA_{p,m}(G)$.
\end{proof}

Our aim is to highlight several group-theoretic properties allowing to  transfer such properties between commensurable groups.\footnote{Two group $G_1$ and $G_2$ are commensurable if they have isomorphic finite-index subgroups.}

\begin{fact}\label{F:important facts}
    Consider the following properties for an infinite group $G$:
    \begin{list}{$\bullet$}{}
        \item $G$ is finitely generated.
        \item $G$ is linear (i.e. there is a faithful representation of $G$ into $GL_n(F)$ for some field $F$).
        \item $G$ is virtually solvable (equivalently, solvable-by-finite).
    \end{list}
    Then
    \begin{enumerate}
        \item Finite generation is preserved under commensurability.  
        \item Linearity is preserved under commensurability.
        \item Virtual solvability is preserved under commensurability.
        \item Tits alternative: Assuming $G$ is finitely generated and linear, virtual solvability is equivalent to the nonexistence of a non-abelian free subgroup of $G$.
        \item If $G$ is finitely generated, linear and not virtually solvable then $G^n$ is not generic in $G$ for all $n\geq 2$.
    \end{enumerate}
\end{fact}
\begin{proof}
    (1) This is classical. 
    (2) Any subgroup of a linear group is linear. For the other direction, use \cite[Lemma 2.3]{We:lineargroups}.
    (3) By definition.
    (4) \cite{Ti:alternative}.
    (5) \cite[Theorem B]{HrKrLuSh:Powers}.
\end{proof}

\begin{definition}
    A group is said to have unique root extraction if  $G=A_{n,1}(G)$ for all $n\geq 2$.
\end{definition}
\begin{remark}
    Every biorderable group has unique root extraction.
\end{remark}

\begin{proposition}
    Let $G$ be a finitely generated, linear group which is not virtually solvable. If $G$ has unique root extraction, then every group commensurable with $G$ has property $Q$ and hence is not superstable.
\end{proposition}
\begin{proof}
    Let $\widetilde G$ be commensurable with $G$ and let $H$ be a common finite-index subgroup. By Fact \ref{F:important facts}, $\widetilde G$ is also finitely generated, linear and not virtually solvable hence $\widetilde G^n$ is not generic in $\widetilde G$ for every $n$.

    Since $G$ has unique root extraction, so does $H$. The normal core $\mathrm{Core}_{\widetilde G}(H)$ is a finite-index subgroup of $H$ which is normal in $\widetilde{G}$. It thus also has unique root extraction, and by Lemma \ref{L:basic property Q}(3), $A_{p,1}(\widetilde G)$ is generic in $\widetilde G$   for every sufficiently large prime $p$. It follows that $\widetilde G$ satisfies property $Q$ and that $\widetilde G$ is not superstable by Proposition \ref{P:Prop Q is not superstable}.
\end{proof}
\begin{remark}
    The fact that $G$ satisfies unique root extraction (i.e. $G=A_{n,1}(G)$ for all $n$) was used in order to propagate the result to commensurable groups. If instead of unique root extraction, $G=A_{n,1}(G)$ for some $n\geq 2$ (and $G^n$ is not generic in $G$) then $G$ is not superstable.
\end{remark}

\begin{corollary}
    Any group which is commensurable with one of the following is not superstable:
    \begin{enumerate}
        \item Non-abelian right-angled Artin groups.
        \item Surface groups $\pi_1(\Sigma_g)$ for $g\geq 2$.
        \item Pure braid groups $P_n$ for $n\geq 3$.
    \end{enumerate}
\end{corollary}
\begin{proof}
    (1) Right-angled Artin groups are finitely generated by definition, linear \cite[Corollary 3.6]{HsWi:linear-graph}, not virtually solvable by \cite[Theorem 1.2]{Ba:semifree} and Tits alternative, and have unique root extraction because they are biorderable as a graph product of copies of $\mathbb{Z}$ \cite{Ch:graphproductsorder}.

    (2)  Surface groups are finitely generated by definition, it is classical that they are linear (there are even many faithful representations \cite{DeKe:surfacelinear}), not virtually solvable by \cite[Second footnote]{Pu:commutator} and Tits alternative, and have unique root extraction because they are biorderable \cite[Theorem 3.11]{ClRo:order}.

    (3) Pure braid groups are finitely generated \cite[Corollary 1.19]{KaTu:Braid}, linear \cite[Theorem 3.18]{KaTu:Braid}, not virtually solvable \cite[Theorem 1.16]{KaTu:Braid} and Tits alternative, and have unique root extraction because they are biorderable \cite[Theorem 7.8]{KaTu:Braid}.
\end{proof}

\bibliographystyle{alpha}
\bibliography{stabgroups}

\end{document}